\documentclass[preprint]{elsarticle}

\usepackage{amssymb,amsmath,amsthm, color}
\usepackage[margin=2.8cm]{geometry}
\usepackage{url}
\newtheorem{theorem}{Theorem}[section]
\newtheorem{lemma}[theorem]{Lemma}
\newtheorem{define}[theorem]{Definition}
\newtheorem{cor}[theorem]{Corollary}
\newtheorem{prop}[theorem]{Proposition}
\newtheorem{remark}[theorem]{Remark}
\newtheorem{example}[theorem]{Example}
\newcommand{\lcm}{\mathrm{lcm}}
\newcommand{\F}{\mathbb F}

\newcommand{\N}{\mathbb N}
\newcommand{\Z}{\mathbb Z}
\newcommand{\C}{\mathcal C}
\newcommand{\cyc}{\mathrm{Cyc}}

\newcommand{\rad}{\mathrm{rad}}

\newcommand{\ord}{\mathrm{ord}}

\newcommand{\doublespace}

\begin{document}

\begin{frontmatter}

\title{Permutations from an arithmetic setting}

\author[USP]{Lucas Reis\fnref{fn1}\corref{cor1}}

\ead{lucasreismat@gmail.com}
\address[USP]{Universidade de S\~{a}o Paulo, Instituto de Ci\^{e}ncias Matem\'{a}ticas e de Computa\c{c}\~{a}o, S\~{a}o
Carlos, SP 13560-970, Brazil.}
\fntext[fn1]{Permanent address: Departamento de Matem\'{a}tica, Universidade Federal de Minas Gerais, UFMG, Belo Horizonte, MG, 30123-970, Brazil.}
\author[UFOP]{S\'{a}vio Ribas}
\ead{savio.ribas@ufop.edu.br}
\cortext[cor1]{Corresponding author}

\address[UFOP]{Universidade Federal de Ouro Preto, Instituto de Ci\^{e}ncias Exatas e Biol\'ogicas, Departamento de Matem\'{a}tica, Ouro Preto, MG 35400-000, Brazil.}
\journal{Elsevier}
\begin{abstract}
Let $m, n$ be positive integers such that $m>1$ divides $n$. In this paper, we introduce a special class of piecewise-affine permutations of the finite set $[1, n]:=\{1, \ldots, n\}$ with the property that the reduction $\pmod m$ of $m$ conse\-cutive elements in any of its cycles is, up to a cyclic shift, a fixed permutation of $[1, m]$. Our main result provides the cycle decomposition of such permutations. We further show that such permutations give rise to permutations of finite fields. In particular, we explicitly obtain classes of permutation polynomials of finite fields whose cycle decomposition and its  inverse are explicitly given.
 
\end{abstract}

\end{frontmatter}





\section{Introduction}\label{sec:Intro}

Let $m, n$ be positive integers such that $m>1$ divides $n$. For integers $1 \le k_1 < k_2$, set $[k_1,k_2] = \{t \in \N \mid k_1 \le t \le k_2\}$. In this paper, we introduce a special class of {\em piecewise-affine permutations} of the set $[1, n]$. These permutations are piecewisely defined by affine-like rules, according to classes modulo $m$, in a way that the reduction modulo $m$ of $m$ consecutive elements in any of its cycles is, up to a cyclic shift, a fixed permutation of $[1,m]$. In particular, every cycle of this kind of permutation has length divisible by $m$. One of our main results, Theorem~\ref{thm:cycle}, provides the explicit cycle decomposition of such permutations. We also provide complete results on the characterization and number of such permutations. In particular, we show that the inverses of such permutations are of the same type and can be easily computed.



We further use our piecewise-affine permutations in the construction of permutation polynomials over finite fields.  Namely, let $q$ be a prime power and $\F_q$ be the finite field with $q$ elements. A polynomial $f \in \F_q[x]$ is called a {\em permutation polynomial} if the evaluation map $c\mapsto f(c)$ is a permutation of $\F_q$. It is well known that $\F_q^*$  is a multiplicative cyclic group of order $q-1$. Let $\theta_q$ be a generator of $\F_q^*$. It turns out that if $f: \F_q \to \F_q$ is the function given by
$$f(0) = 0 \quad \text{ and } \quad f(\theta_q^i) = \theta_q^{\pi(i)} \quad \text{ for all } 1 \le i \le q-1,$$
where $\pi$ is a piecewise-affine permutation of $[1,q-1]$, the polynomial representation of the permutation $f$ as well as its cycle decomposition and  inverse can be derived. The permutations like the previous one are piecewise defined by monomials, according to {\em cyclotomic cosets}. This kind of permutations was previously explored in full generality by Wang~\cite{Wa}. However, there is no study on their cycle decomposition. It is worth mentioning that, for only few families of permutation polynomials, we know the cycle decomposition without needing to describe the whole permutation; namely, {\em monomials}~\cite{A69}, {\em Möbius maps}~\cite{CMT08},  {\em Dickson polynomials}~\cite{LM} and certain {\em linearized polynomials}~\cite{MV88, PR18}. 



The idea of bringing piecewise permutations to obtain permutation polynomials was earlier used by Fernando \& Hou \cite{FH} and by Cao, Hu \& Zha \cite{CHZ}, who obtained families of permutation polynomials via certain powers of linearized polynomials and using a matrix approach, respectively. Some other algebraic-combinatorial methods to produce large classes of permutation polynomials include {\em linear translators}~\cite{K11}, {\em algebraic curves}~\cite{BGQZ}, and, most notably, the {\em AGW criterion}~\cite{W2}. See~\cite[\S8]{MP} for more details on permutation polynomials over finite fields and~\cite{Ho} for a survey on recent advances.

The structure of the paper is given as follows. In Section~\ref{sec:pap}, we introduce our class of piecewise-affine permutations of $[1,n]$ and present some fundamental results, including the  inverses (which are also piecewise-affine); in particular, we explore a special subclass of these permutations that are defined by two rules. In Section~\ref{sec:cycdecomp}, we obtain an explicit description on their cycle decomposition. In Section~\ref{sec:pp} we show how these permutations can be used in the construction of permutation polynomials over finite fields and their inverses, and discuss further issues on these permutation polynomials.

\section{On piecewise-affine permutations of the set $[1, n]$}\label{sec:pap}

We start fixing some notation. The letters $n, m$ always denote positive integers such that $m>1$ divides $n$. In general, $\vec{a}$ denotes an $m$-tuple of integers in a fixed range (usually $[1, m]$ or $[1, n]$). Also, $a_i$ denotes the $i$-th coordinate of $\vec{a}$. In addition, for a positive integer $N>1$, let $\rad(N)$ denote the product of the distinct prime divisors of $N$, and let $\rad(1)=1$. We also denote by $\varphi$ the Euler's totient function, and by $\ord_k r$ the order of $r$ modulo $k$.


\begin{define}
Let $\C(m)$ denote the subset of $[1, m]^m$ of the vectors $\vec{c}$ whose entries comprise a permutation of the set $[1, m]$.
\end{define}

\begin{define}
For an integer $k>1$, let $\Psi_k: \mathbb N \to [1,k]$ such that $\Psi_k(a) = a \pmod k$.
\end{define}

\begin{define}\label{nmpap}
An \emph{$(n, m)$-piecewise affine permutation} (or $(n, m)$-p.a.p.)  is a permutation $\pi$ of the set $[1, n]$ such that there exist $\vec{a}, \vec{b}\in [1, n]^m$ and $\vec{c}\in \C(m)$ with the property that
\begin{equation}\label{pipap}
\pi(x)=\Psi_n(a_ix+b_i)\quad \text{and} \quad \Psi_m(\pi(x))=c_{i+1},
\end{equation}
for any $x\in [1, n]$ with $\Psi_m(x)=c_i$, where the indexes are taken modulo $m$. In this case, we say that the triple $(\vec{a}, \vec{b}, \vec{c})$ is $(n, m)$-admissible and $\pi$ is the $(n, m)$-p.a.p. with parameters $(\vec{a}, \vec{b}, \vec{c})$. Furthermore, we say that two $(n, m)$-admissible triples $(\vec{a}, \vec{b}, \vec{c})$ and $(\vec{A}, \vec{B}, \vec{C})$ are $(n, m)$-equivalent if they induce the same permutation on $[1, n]$.
\end{define}

\begin{example}\label{ex:1}
Let $n=12$, $m=3$ and let $\pi$ be the $(12, 3)$-p.a.p. with parameters $(\vec{a}, \vec{b}, \vec{c})$, where $\vec{a}=(1, 3, 5)$, $\vec{b}=(4, 6, 1)$ and $\vec{c}=(1, 2, 3)$. In other words, for each $x\in [1, 12]$,
$$\pi(x)=\begin{cases}\Psi_{12}(x+4)&\text{if}\;\; x\equiv 1\pmod 3,\\ \Psi_{12}(3x+6)&\text{if}\;\; x\equiv 2\pmod 3, \\ \Psi_{12}(5x+1)&\text{if}\;\; x\equiv 0\pmod 3.\end{cases}$$
The cycle decomposition of $\pi$ is given by $\, (1\,\, 5\,\, 9\,\, 10\,\, 2\,\, 12) \,\, (3\,\, 4\,\, 8\,\, 6\,\, 7\,\, 11)$.
\end{example}

In the following theorem we characterize, up to $(n, m)$-equivalence, all the $(n, m)$-admissible triples.

\begin{theorem}\label{(a,n/m)=1}
Let $m>1$ be a divisor of $n$ and write $n=n_1n_2$, where $\rad(n_1)$ divides $\frac{n}{m}$ and $\gcd\left(n_2, \frac{n}{m}\right)=1$. Let $\vec{a}\in [1, n]^m$ and $\vec{c}\in \C(m)$. Then there exists an element $\vec{b}\in [1, n]^m$ such that the triple $(\vec{a}, \vec{b}, \vec{c})$ is $(n, m)$-admissible if and only if the entries of $\vec{a}$ are relatively prime with $n_1$. In this case, there are $(\frac nm)^m$ choices for $\vec{b}$. Moreover, two $(n, m)$-admissible triples $(\vec{a}, \vec{b}, \vec{c})$ and $(\vec{A}, \vec{B}, \vec{C})$ are $(n, m)$-equivalent if and only if there exists $t\in [1, m]$ such that, for every $1\le i\le m$, the following properties hold:
\begin{enumerate}[(i)]
\item $c_i=C_{i+t}$;
\item $a_i\equiv A_{i+t}\pmod {n/m}$;
\item $B_{i+t}=\Psi_n(a_ic_i+b_i-A_{i+t}C_{i+t})$.
\end{enumerate}
\end{theorem}

\begin{proof}
We observe that $\pi$ is an $(n,m)$-p.a.p. with parameters $(\vec{a},\vec{b},\vec{c})$ if and only if $\pi$ is of the form given by Eq.~\eqref{pipap} and $\pi$ is one to one. Suppose that $x,y \in [1,n]$ are such that $\pi(x) = \pi(y)$. Then $\pi(x) \equiv \pi(y) \equiv c_{i+1} \pmod m$ for some $i \in [1,m]$ and, by definition, $x \equiv y \equiv c_i \pmod m$, where $i$ is taken modulo $m$. Therefore, the condition $\pi(x) = \Psi_n(a_ix + b_i) = \Psi_n(a_iy + b_i) = \pi(y)$ is equivalent to $a_i(x-y) \equiv 0 \pmod n$. Since $n_2$ divides $n$ and $\gcd\left(n_2,\frac{n}{m}\right) = 1$, we have that $n_2$ divides $m$. Since $\gcd(n_1,n_2)=1$, $x \equiv y \pmod m$ and $n_2$ divides $m$, the equation $a_i(x-y) \equiv 0 \pmod n$ is equivalent to 
$$a_i(x-y) \equiv 0 \pmod {n_1},$$
which has the unique solution $x \equiv y \pmod {n_1}$ if and only if $\gcd\left( a_i , n_1 \right) = 1$. In particular, $\pi$ is one to one if and only if $\gcd\left( a_i , n_1 \right) = 1$ for any $i\in [1, m]$. In this case, each $b_i$ is uniquely determined modulo $m$ by $$b_i \equiv c_{i+1} - a_ic_i \pmod m.$$ 
Since the entries of $\vec b$ lie in $[1,n]$, there exist $(\frac nm)^m$ possibilities for $\vec b$.

Moreover, if the $(n,m)$-admissible triples $(\vec{a},\vec{b},\vec{c})$ and $(\vec{A},\vec{B},\vec{C})$ are $(n,m)$-equivalent, then, up to a cyclic shift, $\vec{c}$ and $\vec{C}$ are the same, so $(i)$ holds. From this, we obtain $\Psi_n(a_ic_i + b_i) = \Psi_n(A_{i+t}C_{i+t} + B_{i+t})$, which implies $(iii)$. Furthermore, for every $i \in [1,m]$ and every $j \in [1,\frac nm]$, we have that $a_i(c_i+jm)+b_i \equiv A_{i+t}(C_{i+t}+jm)+B_{i+t} \pmod n$, which implies that $(a_i-A_{i+t})j \equiv 0 \pmod {\frac nm}$, proving $(ii)$. Conversely, if $(i)$, $(ii)$ and $(iii)$ hold, then for every $i \in [1,m]$ and $j \in [1,\frac nm]$,  the identities $a_i(c_i+jm)+b_i \equiv A_{i+t}(C_{i+t}+jm)+B_{i+t} \pmod n$ and $\Psi_m(a_i(c_i+jm)+b_i) = c_i = C_{i+t} = \Psi_m(A_{i+t}(C_{i+t}+jm)+B_{i+t})$ hold. Therefore, $(\vec{a}, \vec{b}, \vec{c})$ and $(\vec{A},\vec{B},\vec{C})$ are $(n,m)$-equivalent.

\end{proof}

From the previous theorem, we obtain the exact number of permutations arising from $(n, m)$-p.a.p.'s.

\begin{cor}\label{cor:numberpap}
Let $m>1$ be a divisor of $n$ and write $n=n_1n_2$, where $\rad(n_1)$ divides $\frac{n}{m}$ and $\gcd\left(n_2, \frac{n}{m}\right)=1$. Then the number of distinct $(n, m)$-p.a.p.'s equals
$$(m-1)!\cdot \left(\frac{n\cdot n_2\cdot \varphi(n_1)}{m^2}\right)^m.$$
\end{cor}

\begin{proof}
First, we compute the number of $(n, m)$-admissible triples. There are $m!$ choices for $\vec c$, and $n_2\cdot \varphi \left( n_1\right)$ choices for each $a_i$, hence $\left[ n_2 \cdot \varphi \left( n_1 \right) \right]^m$ choices for $\vec a$. In addition, for fixed $\vec a$ and $\vec c$, there are $\left( \frac nm \right)^m$ choices for $\vec b$. Therefore, the number of $(n, m)$-admissible triples equals $m!\cdot \left(\frac{n \cdot n_2 \cdot \varphi(n_1)}{m}\right)^m$. For a fixed $(n, m)$-admissible triple $(\vec{a}, \vec{b}, \vec{c})$, Theorem~\ref{(a,n/m)=1} entails that such triple is $(n, m)$-equivalent to exactly $m\cdot \delta_m$  $(n, m)$-admissible triples, where $\delta_m$ is the number of ways of choosing vectors $(u_1, \ldots, u_m)\in [1, n]^m$ with $\gcd(u_i, n_1)=1$ and $u_i\equiv a_i\pmod {n/m}$. From contruction, $\rad(n_1)=\rad(n/m)$ and so  $u_i\equiv a_i\pmod {n/m}$ already implies that $\gcd(u_i, n_1)=1$. Hence, $\delta_m=m^m$ and the result follows.

\end{proof}

\subsection{On the inverse of $(n, m)$-p.a.p.'s}

We show that the inverse of an $(n, m)$-p.a.p. is another $(n, m)$-p.a.p. whose parameters can be explicitly computed (though not unique).

\begin{theorem}\label{thm:inverse}
Let $m>1$ be a divisor of $n$ and write $n=n_1n_2$, where $\rad(n_1)$ divides $\frac{n}{m}$ and $\gcd\left(n_2, \frac{n}{m}\right)=1$. Let $(\vec{a}, \vec{b}, \vec{c})$ be an $(n, m)$-admissible triple and $\pi$ be the $(n, m)$-p.a.p. with parameters $(\vec{a}, \vec{b}, \vec{c})$. For each $1\le i\le m$, let $A_i\in [1, n_1]$ be such that $A_i\cdot a_{i-1}\equiv 1\pmod {n_1}$. Then, for each $1\le i\le m$, the system of congruences 
\begin{equation}\label{eq:systemb}
\begin{cases}
x\equiv c_{i-1}-A_ic_i\pmod m\\ x\equiv -A_ib_{i-1}\pmod {n_1},
\end{cases}
\end{equation}
admits a solution $B_i\in [1, n]$. Also, if $\vec{A}=(A_m, A_{m-1}, \ldots, A_1)$, $\vec{B}=(B_m, B_{m-1}, \ldots, B_1)$ and $\vec{C}=(c_m, c_{m-1}, \ldots, c_1)$, then the triple $(\vec{A}, \vec{B}, \vec{C})$ is $(n, m)$-admissible and the permutation $\pi^{-1}$ induced by such triple is the inverse of $\pi$, i.e.,
$$\pi(\pi^{-1}(y))=\pi^{-1}(\pi(y))=y, \;\;\; y\in [1, n].$$
\end{theorem}

\begin{proof}
From Theorem~\ref{(a,n/m)=1}, we have that $\gcd(a_i, n_1)=1$ for every $1 \le i \le m$, and so $A_i$ is well defined. 
In order to prove that the system above has a solution $B_i \in [1,n]$, the Chinese Remainder Theorem entails that it suffices to show that $$-A_ib_{i-1}\equiv c_{i-1}-A_ic_i\pmod {\gcd(m, n_1)},$$
for every $1\le i\le m$. This is true since the congruence $a_{i-1}c_{i-1}+b_{i-1} \equiv c_i\pmod m$ implies that $c_{i-1}+A_ib_{i-1} \equiv A_ic_i \pmod {\gcd(m,n_1)}$. From definition, $\gcd(A_i, n_1)=1$ and 
$$A_ic_i+B_i\equiv c_{i-1}\pmod m.$$
From Theorem~\ref{(a,n/m)=1}, the triple $(\vec{A}, \vec{B}, \vec{C})$ is $(n, m)$-admissible. So it remains to prove that $\pi(\pi^{-1}(y))=\pi^{-1}(\pi(y))=y$ for  $y\in [1, n]$. We only prove $\pi^{-1}(\pi(y))=y$ since the equality $\pi(\pi^{-1}(y))=y$ follows in a similar way. From definition, $n_1$ and $n_2$ are relatively prime and $n_2$ divides $m$. So it suffices to prove that $\pi^{-1}(\pi(y))\equiv y\pmod t$ for $t \in \{n_1, m\}$. Suppose that $y\equiv c_i\pmod m$, hence $\pi(y)\equiv c_{i+1}\pmod m$ and so $\pi^{-1}(\pi(y))\equiv c_i\equiv y\pmod m$. Moreover, $\pi^{-1}(\pi(y))=\Psi_n(A_{i+1}(a_iy+b_i)+B_{i+1})$. Recall that $n_1$ is a divisor of $n$. From $A_{i+1}a_i\equiv 1\pmod {n_1}$ and $B_{i+1}\equiv -A_{i+1}b_i\pmod {n_1}$, we conclude that 
$$\Psi_n(A_{i+1}(a_iy+b_i)+B_{i+1})\equiv y\pmod {n_1}.$$
\end{proof}

\begin{example}\label{ex:12,3}
Let $\pi$ the $(12,3)$-p.a.p. defined in Example~\ref{ex:1}. 
Then the  inverse $\pi^{-1}$ of $\pi$ is the $(12, 3)$-p.a.p. with parameters $(\vec{A}, \vec{B}, \vec{C})$, where
$\vec{A}=(3, 1, 1)$, $\vec{B}=(2, 8, 11)$ and $\vec{C}=(3, 2, 1)$ so that, for each $x\in [1, 12]$,
$$\pi^{-1}(x)=
\begin{cases}
\Psi_{12}(x+11)&\text{if}\;\; x\equiv 1\pmod 3,\\ 
\Psi_{12}(x+8)&\text{if}\;\; x\equiv 2\pmod 3, \\ 
\Psi_{12}(3x+2)&\text{if}\;\; x\equiv 0\pmod 3.
\end{cases}$$

\end{example}

\subsection{The class of p.a.p.'s defined by two rules}

Here we  introduce the special class of $(n, m)$-p.a.p.'s that can be defined by two affine-like rules, one for the multiples of $m$ and one for the remaining integers in $[1, n]$. More specifically, we have the following definition.
 
\begin{define}\label{def:2-}
An $(n, m)$-p.a.p. $\pi$ is said to be {\em $2$-reducible} if there exist integers $a_0, a, b_0, b\in [1, n]$ such that, for any $x\in [1, n]$, we have that
$$\pi(x)=
\begin{cases}
\Psi_n\left(a_0\cdot x+b_0\right) & \text{if}\; x\equiv 0 \pmod m,\\
\Psi_n\left(a\cdot x+b\right)& \text{if}\; x\not\equiv 0 \pmod m.
\end{cases}$$
In this case, the quadruple $(a_0, a, b_0, b)$ is called the \emph{$2$-reduced parameters of $\pi$}.
\end{define}

From definition, any $(n, m)$-p.a.p. is $2$-reducible if $m=2$. Our aim is to provide a complete characterization of the $2$-reducible $(n, m)$-p.a.p.'s, where $m>2$. We start with the following auxiliary lemmas.

\begin{lemma}[Lifting the Exponent Lemma]\label{lem:lel}
Let $p$ be a prime and $\nu_p$ be the $p$-valuation. The following hold:
\begin{enumerate}[(1)]
\item if  $p$ is an odd prime divisor of $a-1$, $\nu_{p}(a^{k}-1)=\nu_{p}(a-1)+\nu_{p}(k)$;
\item if $p=2$  and $a>1$ is odd,
$$
\nu_{2}(a^{k}-1)=
\begin{cases}
\nu_{2}(a-1)&\text{if $k$ is odd,} \\
\nu_{2}(a^{2}-1)+\nu_{2}(k)- 1&\text{if $k$ is even.}\\
\end{cases}
$$
\end{enumerate}
\end{lemma}

\begin{lemma}\label{b(a...)}
Let $a, b, m$ be positive integers and set $\rad_2(m)=\rad(m)\cdot \gcd(m, 2)$. Then the reductions modulo $m$ of the numbers $$b, \; b(a+1), \; \ldots, \; b(a^{m-1}+\cdots+a+1)$$ are all distinct if and only if $\gcd(b, m)=1$ and $a\equiv 1\pmod {\rad_2(m)}$.
\end{lemma}

\begin{proof}
Set $f_0=b$ and, for $1\le i\le m-1$, set $f_i=b(a^i+\cdots+a+1)$.
It is clear that $b$ (resp. $a$) must be relatively prime with $m$, since otherwise the reduction modulo $m$ of the elements $f_i$ would not contain the class $1$ (resp. the class $0$). In particular, for $0\le i<j\le m-1$, $f_i\equiv f_j\pmod m$ if and only if $f_{j-i}\equiv 0\pmod m$. Therefore, it suffices to prove that $i=m-1$ is the smallest index such that $f_i\equiv 0\pmod m$ if and only if $a\equiv 1\pmod{\rad_2(m)}$. Of course, this holds for $a=1$. Suppose that $a>1$ and write $m=m_0m_1$, where $\rad(m_0$) divides $a-1$ and $m_1$ is relatively prime with $a-1$. In other words, we want to prove that $\ord_{m(a-1)}a=m$ if and only if $m_0=m$ and $a\equiv 1\pmod 4$ if $m$ is even. Since $m_1$ and $a-1$ are relatively prime, we have that
\begin{equation}\label{eq:orders}
\ord_{m(a-1)}a=\mathrm{lcm}(\ord_{m_0(a-1)}a, \ord_{m_1} a)\le \ord_{m_0(a-1)}a\cdot \ord_{m_1}a,
\end{equation}
with equality if and only if $\ord_{m_0(a-1)}a$ and $\ord_{m_1}a$ are relatively prime. However, from Lemma~\ref{lem:lel}, $\ord_{m_0(a-1)}a\le m_0$ with equality if and only if $m_0$ is odd or $m_0$ is even and $a\equiv 1\pmod 4$. In addition, $\ord_{m_1}a\le \varphi(m_1)<m_1$ whenever $m_1>1$. Therefore, from Eq.~\eqref{eq:orders}, we have that $\ord_{m(a-1)}a=m$ if and only if $m_1=1$ (i.e., $m_0=m$) and $a\equiv 1\pmod 4$ if $m$ is even.

\end{proof}

In the following proposition we describe the $2$-reducible $(n, m)$-p.a.p's.
\begin{prop}\label{prop:2-reducible}
Let $m>2$ be a positive divisor of $n$ and write $n=n_1n_2$, where $\rad(n_1)$ divides $\frac{n}{m}$ and $\gcd\left(n_2, \frac{n}{m}\right)=1$. For integers $a_0, a, b_0, b\in [1, n]$, the quadruple $(a_0, a, b_0, b)$ provides the $2$-reduced parameters of a $2$-reducible $(n, m)$-p.a.p. $\pi$ if and only if the following properties hold:
\begin{enumerate}[(i)]
\item $b$ and $b_0$ are relatively prime with $m$, and $b \equiv b_0 \pmod m$;
\item $a\equiv 1\pmod {\rad_2(m)}$;
\item $a$ and $a_0$ are relatively prime with $n_1$.
\end{enumerate}
In this case, $\pi$ is the $(n, m)$-p.a.p. with parameters $(\vec{a}, \vec{b}, \vec{c})$, where $\vec{a}=(a_0, a, \dots a)$, $\vec{b}=(b_0, b, \dots, b)$ and $\vec{c}$ is a cyclic permutation of the vector $(c_1, \ldots, c_m)$ with
\begin{equation}\label{c_i}
c_i=
\begin{cases}
m & \text{if}\;\;\; i=1,\\
\Psi_m(b \cdot (a^{i-2}+a^{i-3}+\ldots+a+1))& \text{if}\;\;\; 2\le i\le m.
\end{cases}
\end{equation}
Moreover, the inverse $\pi^{-1}$ of the $2$-reducible $(n,m)$-p.a.p. $\pi$ is defined by the following affine rules:
$$\pi^{-1}(x) = 
\begin{cases}
\Psi_n(A_0 \cdot x + B_0) \quad \, \text{ if $x \equiv b \pmod m$} \\
\Psi_n(A \cdot x + B) \quad \quad \text{ if $x \not\equiv b \pmod m$}
\end{cases}$$
where $A_0 \equiv a_0^{-1} \pmod {n_1}$, $A \equiv a^{-1} \pmod {\lcm(m,n_1)}$, $B_0$ is a solution of the system~\eqref{eq:systemb} with $i=2$ and $B = \Psi_n(-Ab)$.

\end{prop}

\begin{proof}
For the first part, we just need to show that, if $(a_0,a,b_0,b)$ are the $2$-reduced parameters of the $2$-reducible $(n,m)$-p.a.p. $\pi$, then we necessarily have that $b \equiv b_0 \pmod m$. The remainder `if and only if' part follows from Theorem~\ref{(a,n/m)=1} and Lemma~\ref{b(a...)}; the further identities for $c_i$ follow directly by calculations. 

We observe that, since $\pi$ is an $(n, m)$-p.a.p., for any $t\in [1, m]$ such that $t\ne \Psi_m(b_0)$, there exists $y=y(t)\in [1, m-1]$ such that $ay+b\equiv t\pmod m$. In particular, if $b\not\equiv b_0\pmod m$, there exists $y\in [1, m-1]$ such that $ay+b\equiv b\pmod m$ and so $ay\equiv 0\pmod m$. This implies that $d:=\gcd(a, m)>1$. However, in this case, the set $\{\Psi_m(ay+b) \mid y\in [1, m-1]\}$ has at most $\frac{m}{d}$ elements. Since $m>2$, we have that $\frac{m}{d}<m-1$ and so we get a contradiction with the property of $y(t)$.

The expression for the parameters $A_0, A \in [1,n_1]$ and $B_0 \in [1,\lcm(m,n_1)]$ of $\pi^{-1}$ follows directly from Theorem~\ref{thm:inverse}. Furthermore, we can extend $A \in [1,\lcm(m,n_1)]$ to be also the inverse of $a$ modulo $m$ so that $A \equiv a^{-1} \pmod {\lcm(m,n_1)}$, since $\gcd(a,m) = 1$ by item (ii). Let $B = \Psi_n (-Ab)$. We are going to show that $B$ is a solution of the system~\eqref{eq:systemb} for every $i \in [1,m]\backslash\{2\}$. The second equation of system~\eqref{eq:systemb} is trivial. If $i \in [3,m]$ then the first one is equivalent to $B \equiv b(a^{i-3} + \dots + a + 1) - Ab(a^{i-2} + \dots + a + 1) \pmod m$, which is true since $A \equiv a^{-1} \pmod m$. If $i=1$ then the first equation of~\eqref{eq:systemb} is equivalent to $-Ab \equiv b(a^{m-2} + \dots + a + 1) \pmod m$, which follows from Lemma~\ref{b(a...)}.

\end{proof}

From the previous proposition, a lower bound for the number of $2$-reducible $(n, m)$-p.a.p.'s is derived.


\begin{cor}
The number of non-equivalent $2$-reducible $(n, m)$-p.a.p.'s is at least
$$\varphi\left(\frac{n}{m}\right)\cdot \frac{\varphi(m)\cdot n^2}{m^2}.$$
\end{cor}

\begin{proof}
We provide a class of  non-equivalent $(n, m)$-p.a.p.'s with $2$-reduced parameters of the form $(a_0, 1, b_0, b)$ which proves the claim. Let $C$ be the set of quadruples $(a_0, 1, b_0, b)$ such that $1 \le a_0\le \frac{n}{m}$, $\gcd\left(a_0, \frac{n}{m}\right)=1$ (hence $\gcd(a_0, n_1)=1$), $1 \le b_0\le n$, $\gcd(b, m)=1$ and $b_0\equiv b\pmod m$.  Proposition~\ref{prop:2-reducible} entails that any element of $C$ yields a $2$-reducible $(n, m)$-p.a.p. and it is clear that $C$ has exactly $\varphi\left(\frac{n}{m}\right)\cdot \frac{\varphi(m)\cdot n^2}{m^2}$ elements. We just need to verify that any two of them yield non-equivalent permutations of $[1, n]$. Suppose that two elements $(a_0, 1, b_0, b)$ and $(a_0', 1, b_0', b')$ of $C$ yield the same permutation of $[1, n]$. Since $a_0, a_0'\le \frac{n}{m}$, Theorem~\ref{(a,n/m)=1} entails that $a_0=a_0'$. Also, taking $x=n$ in Definition~\ref{def:2-}, we have that $b_0=b_0'$ and the same definition readily implies that $b=b'$.



\end{proof}

\section{Cycle decomposition}\label{sec:cycdecomp}
We fix $(\vec{a}, \vec{b}, \vec{c})$ an $(n, m)$-admissible triple and $\pi=\pi(\vec{a}, \vec{b}, \vec{c})$ the $(n, m)$-p.a.p. with parameters $(\vec{a}, \vec{b}, \vec{c})$. The following proposition provides basic properties of the cycle decomposition of $\pi$.

\begin{prop}\label{prop:basic}
For any $y\in [1, n]$, the following properties hold:
\begin{enumerate}[(i)]
    \item the cycle of $\pi$ containing $y$ has length divisible by $m$;
    \item there exists an element $z\in [1, n]$ such that $z$ is divisible by $m$ and lies in the same cycle of $\pi$ containing $y$.
\end{enumerate}
In addition, if a cycle of $\pi$ has length $mt$, then for each $i\in [1, m]$, such a cycle contains exactly $t$ elements congruent to $i$ modulo $m$. 
\end{prop}

\begin{proof}
\begin{enumerate}[(i)]
   \item From Definition~\ref{nmpap}, $\Psi_m(\pi(x)) = c_{i+1}$ whenever $\Psi_m(x) = c_i$. This guarantees that the sequence $$\Psi_m(y), \; \Psi_m(\pi(y)), \; \Psi_m(\pi^{(2)}(y)), \; \ldots$$ can only return to $\Psi_m(y)$ after cyclically running through the entries of $\vec c \in \mathcal C(m)$. In particular, the sequence $$y, \; \pi(y), \; \pi^{(2)}(y), \; \ldots$$ 
   has minimal period  divisible by $m$.
   \item In fact, there is an entry of $\vec c$ equals to $m$, and its correspondent in the above sequence is divisible by $m$.
\end{enumerate}
We observe that, in a cycle of length $mt$ of $\pi$, $\vec c$ is traversed $t$ times if we consider the reduction modulo $m$ of its elements. Therefore, each $i \in [1,m]$ appears exactly $t$ times.

\end{proof}

In particular, in order to compute the cycle decomposition of $\pi$, Proposition~\ref{prop:basic} entails that it suffices to compute the minimal period of the multiples of $m$ in the set $[1, n]$. In this context, the following definition is useful.

\begin{define}\label{ppps}
\begin{enumerate}
    \item The {\em principal product} of $\pi=\pi(\vec{a}, \vec{b}, \vec{c})$ is $P_{\pi}=\prod_{i=1}^ma_i$.
    \item The {\em principal sum} of $\pi=\pi(\vec{a}, \vec{b}, \vec{c})$ is the unique positive integer $S_{\pi}\in [1, n]$ with the property that $\pi^{(m)}(x)=\Psi_{n}(P_{\pi}\cdot x+S_{\pi})$,
for any $x\in [1, n]$ such that $x\equiv 0\pmod m$.
\end{enumerate}
\end{define}



\begin{example}
The principal product and principal sum of the $2$-reducible $(n,m)$-p.a.p. $\pi$ with parameters $(a_0,a,b_0,b)$ are $P_{\pi}=a_0a^{m-1}$ and $S_{\pi}=\Psi_n\left(b_0a^{m-1} + b(a^{m-2} + \ldots + a + 1)\right)$ respectively. 
\end{example}

The following lemma provides a way of obtaining the $mk$-th iterates of $\pi$ at elements $x\in [1, n]$ that are divisible by $m$.

\begin{lemma}
The principal sum $S_{\pi}$ of $\pi(\vec{a}, \vec{b}, \vec{c})$ is well defined and, for any 
positive integers $k, x$ such that $x\in [1, n]$ is divisible by $m$, we have that

$$\pi^{(mk)}(x)=\Psi_{n}\left(P_{\pi}^k\cdot x+\frac{P_{\pi}^k-1}{P_{\pi}-1}\cdot S_{\pi}\right),$$
whenever $P_{\pi} \ne 1$. For $P_{\pi} \equiv 1 \pmod n$, we have that $\pi^{(mk)}(x)=\Psi_n(x+k\cdot S_{\pi})$ and, for $P_{\pi} = 1$, we have that $\pi^{(mk)}(x) = \Psi_n\left(x+k\cdot \sum_{1\le i\le m} b_i\right)$.
\end{lemma}

\begin{proof}
The composition of affine functions is also affine. Since $\pi^{(m)}(x)$ is the reduction modulo $n$ of the composition of $m$ affine functions given by Definition~\ref{nmpap}, each of which has slope $a_i$, $\pi^{(m)}(x)$ is affine as well, with slope $P_{\pi}$. Therefore, the linear coefficient $S_{\pi}$ is well-defined. In fact, by reindexing $\vec{c}$ under a cyclic shift if needed, $S_{\pi}$ is given by 
\begin{equation}\label{Spi}
S_{\pi} = \Psi_n \left( \sum_{i=1}^m a_m a_{m-1} \dots a_{i+2} a_{i+1} b_i \right).
\end{equation}
For the remainder, we proceed by induction on $k$. The case $k=1$ follows from the definition of principal sum. Suppose that $$\pi^{(mk)}(x)=\Psi_{n}\Big(P_{\pi}^k\cdot x+ (P_{\pi}^{k-1} + \dots + P_{\pi} + 1)\cdot S_{\pi}\Big)$$ for some $k \ge 1$. Then
\begin{align*}
\pi^{(m(k+1))}(x) &= \pi^{(mk)}(\pi^{(m)}(x)) = \pi^{(mk)}(\Psi_{n}(P_{\pi}\cdot x+S_{\pi})) \\
&= \Psi_{n}\Big(P_{\pi}^k\cdot (P_{\pi}\cdot x+S_{\pi}) + (P_{\pi}^{k-1} + \dots + P_{\pi} + 1)\cdot S_{\pi}\Big) \\
&= \Psi_{n}\Big(P_{\pi}^{k+1} \cdot x+ (P_{\pi}^{k} + P_{\pi}^{k-1} + \dots + P_{\pi} + 1)\cdot S_{\pi}\Big),
\end{align*}
from where we obtain directly the cases $P_{\pi} \ne 1$ and $P_{\pi} \equiv 1 \pmod n$. If $P_{\pi} = 1$ then $a_i = 1$ for all $1 \le i \le m$, which implies $S_{\pi} = \Psi_n\left( \sum_{1 \le i \le m} b_i \right)$.

\end{proof}

Since $\pi^{(m)}(x) \equiv x \pmod m$, it holds $S_{\pi} \equiv 0 \pmod m$. The previous lemma implies the following result.

\begin{prop}\label{prop:length}
Let $\pi$ be an $(n, m)$-p.a.p. with principal product $P_{\pi}$ and principal sum $S_{\pi}$. For any positive integer $x\in [1, n]$ divisible by $m$ with $x=mx_0$, the length of the cycle of $\pi$ containing $x$ is given as follows:

\begin{enumerate}[(i)]
\item $m\cdot \frac{n}{\gcd(n, S_{\pi})}$ if $P_{\pi}\equiv 1\pmod n$;
\item if $P_{\pi}\ne 1$, this length is given by $m\cdot \ord_{\kappa(x)}P_{\pi}$, where
\begin{align} 
\kappa(x) &= \frac{n\cdot (P_{\pi}-1)}{\gcd(n\cdot (P_{\pi}-1),  x\cdot (P_{\pi}-1)+S_{\pi})} = \frac{ \frac{n}{m} \cdot (P_{\pi}-1)}{g_{\pi}\cdot \gcd\left(\frac{n}{m} \cdot \frac{P_{\pi}-1}{g_{\pi}} ,  x_0\cdot \frac{P_{\pi}-1}{g_{\pi}}+\frac{S_{\pi}}{m \cdot g_{\pi}}\right)}, \label{kappa} \\
g_{\pi} &= \gcd\left( \frac{S_{\pi}}{m}, P_{\pi}-1 \right). \label{gpi}
\end{align} 
\end{enumerate}
\end{prop}

\begin{proof}
\begin{enumerate}[(i)]
\item In this case, the cycle has length $mk$ if and only if $k$ is minimal such that $\pi^{(mk)}(x) = \Psi_n(x+k\cdot S_{\pi}) = x$, i.e., $kS_{\pi}\equiv 0\pmod n$. It is clear that the minimal $k$ satisfying the latter equals $\frac{n}{\gcd(n, S_{\pi})}$.
\item In this case, $\pi^{(mk)}(x) = \Psi_n\left(P_{\pi}^k\cdot x+\frac{P_{\pi}^k-1}{P_{\pi}-1}\cdot S_{\pi}\right) = x$, and so we have the following equivalent conditions:
\begin{align*}
(P_{\pi}^k - 1)\cdot x+\tfrac{P_{\pi}^k-1}{P_{\pi}-1}\cdot S_{\pi} &\equiv 0\pmod n, \\
(P_{\pi}^k - 1)\cdot [x \cdot (P_{\pi} - 1) + S_{\pi}] &\equiv 0 \pmod {n(P_{\pi} - 1)}, \\
P_{\pi}^k - 1 &\equiv 0 \pmod {\kappa(x)}.
\end{align*}
Therefore, the smallest possible $k > 0$ is $k = \ord_{\kappa(x)} P_{\pi}$ and the cycle of $\pi$ containing $x$ has length equals $m \cdot \ord_{\kappa(x)} P_{\pi}$.
\end{enumerate}
\end{proof}

The next lemma displays all the possible values of 
\begin{equation}\label{N0}
N_0 = N_0(x_0) := \gcd \left(\frac{n}{m} \cdot \frac{P_{\pi}-1}{g_{\pi}} \; , \; \frac{P_{\pi}-1}{g_{\pi}} \cdot x_0 + \frac{S_{\pi}}{m \cdot g_{\pi}}\right),
\end{equation}
and the number of solutions in each case. By Eqs.~\eqref{kappa} and~\eqref{N0}, we observe that 
\begin{equation}\label{expressaokappa}
\kappa(x) = \dfrac{\frac{n}{m} \cdot (P_{\pi} - 1)}{g_{\pi} \cdot N_0}.
\end{equation}

\begin{lemma}\label{lem:alpha}
Let $\alpha,\beta,\gamma$ be positive integers such that $\gcd(\alpha,\beta) = 1$ and $\alpha$ divides $\gamma$. Write $\gamma = \gamma_1\gamma_2$, where $\rad(\gamma_1)$ divides $\alpha$ and $\gcd(\gamma_2,\alpha) = 1$. Then the following properties hold:
\begin{enumerate}[(i)]
\item as $y$ runs over $[1,\gamma/\alpha]$, $\gcd(\alpha y+\beta,\gamma)$ runs over all the divisors of $\gamma_2$;
\item for each divisor $d$ of $\gamma_2$, the number of solutions $y \in [1,\gamma/\alpha]$ of the equation $$\gcd(\alpha y + \beta, \gamma) = \gamma_2/d$$ is $\varphi(d) \cdot \gamma_1$.
\end{enumerate}
\end{lemma}

\begin{proof}
\begin{enumerate}[(i)]
\item We have that $\alpha$ divides $\gamma_1$. Since $\gcd(\gamma_1,\gamma_2)=\gcd(\alpha, \beta)=1$, we obtain the following equalities $$\gcd(\alpha y+\beta,\gamma) = \gcd(\alpha y+\beta,\gamma_1\gamma_2) = \gcd(\alpha y+\beta,\gamma_2).$$
In particular, $\gcd(\alpha y+\beta,\gamma)$ divides $\gamma_2$. Let $d$ be a positive divisor of $\gamma_2$. In the following, we show that there exists $y \in [1,\gamma/\alpha]$ such that 
$$\begin{cases}
\alpha y + \beta \equiv \beta \pmod {\gamma_1} \\
\alpha y + \beta \equiv d \pmod {\gamma_2}
\end{cases}$$
and this implies that $\gcd(\alpha y+\beta,\gamma_2)=d$. The first congruence is equivalent to $y = t\gamma_1/\alpha$ for some $t \in \Z$, and the second one is equivalent to $t\gamma_1 \equiv d-\beta \pmod {\gamma_2}$, which has a solution for $t \in [1,\gamma_2]$, so that $y \in [1,\gamma/\alpha]$. 

 \item Let $\omega \in [1,\gamma/\alpha]$ be the smallest solution of $\gcd(\alpha y+\beta,\gamma) = \gamma_2/d$. All the other solutions are of the form $\omega+j\frac{\gamma_2}{d}$ with $0\le j<\gamma_1d$. Since $\gcd(\alpha,\gamma_2) = \gcd(\gamma_1,\gamma_2) = 1$, the number $\omega + j \frac{\gamma_2}{d}$ is a solution as well if and only if 
$$\gcd\left( \frac{\alpha \omega + \beta}{\gamma_2/d} + \alpha j , d \right) = 1 \quad \text{ and } \quad 0 \le j < \gamma_1d.$$
Therefore, the number of solutions $x_0 + j \cdot \frac{\gamma_2}{d} \in [1,\gamma/\alpha]$ of this equation is $\varphi(d) \cdot \gamma_1$.
\end{enumerate}
\end{proof}

Suppose that $P_{\pi} > 1$ and let $\alpha = \frac{P_{\pi}-1}{g_{\pi}}$, $\beta = \frac{S_{\pi}}{m \cdot g_{\pi}}$ and $\gamma = \frac{n}{m} \cdot \frac{P_{\pi}-1}{g_{\pi}} = \frac{n}{m}\alpha$ be as in Lemma~\ref{lem:alpha}. Write $\frac{n}{m} = N_1N_2$, where $\rad(N_1)$ divides $\alpha$ and $\gcd\left(N_2,\alpha \right) = 1$. Hence the number $N_0$ defined by Eq.~\eqref{N0} can be any divisor of $N_2$, that is, $N_0$ can be any divisor of $n/m$ that is relatively prime with $\frac{P_{\pi}-1}{g_{\pi}}$ when $x_0$ runs over $[1,n/m]$. This observation and Eq.~\eqref{expressaokappa} easily imply the following result.

\begin{cor}\label{cor:arithmetic}
Fix $x = mx_0 \in [1,n]$. Let $\kappa(x)$ and $N_0$ be defined as in Eqs.~\eqref{kappa} and~\eqref{N0}, respectively. We write $\frac nm = N_1N_2$ where $\rad(N_1)$ divides $\frac{P_{\pi} - 1}{g_{\pi}}$ and $\gcd\left(N_2,\frac{P_{\pi}-1}{g_{\pi}}\right) = 1$, as above. For a divisor $d$ of $N_2$ such that $N_0 = N_2/d$, we have that $\kappa(x) = N_1 \cdot \frac{P_{\pi} - 1}{g_{\pi}} \cdot d$. In addition, the equation $N_0 = N_2/d$ has exactly $\varphi(d)\cdot N_1$ solutions $x$ with $x_0 \in [1,n/m]$. 
\end{cor}

Finally, we exhibit the cycle decomposition of an arbitrary $(n,m)$-p.a.p. $\pi$ with principal product $P_{\pi}\ne 1$ (the case $P_{\pi}\equiv 1\pmod n$ follows trivially by item (i) of Proposition~\ref{prop:length}). In what follows, $\cyc(r)$ denotes a cycle of length $r$. Moreover, $G_1 \oplus G_2$ denotes the disjoint union of the graphs $G_1$ and $G_2$, $\bigoplus_{\ell \in \Lambda} G_\ell$ denotes the disjoint union of the graphs $G_{\ell}$ for $\ell \in \Lambda$ and, for a positive integer $k$, $k \times G = \bigoplus_{1 \le i \le k} G$.

\begin{theorem}\label{thm:cycle}
Let $\pi$ be an $(n,m)$-p.a.p. with principal product $P_{\pi} > 1$ and principal sum $S_{\pi}$. Let $g_{\pi}$ be defined as in Eq.~\eqref{gpi} and write $n/m = N_1N_2$, where $\rad(N_1)$ divides $\frac{P_{\pi}-1}{g_{\pi}}$ and $\gcd\left( \frac{P_{\pi}-1}{g_{\pi}},N_2 \right) = 1$. For each divisor $d$ of $N_2$, set  $\eta(d) = N_1 \cdot \frac{P_{\pi} - 1}{g_{\pi}} \cdot d$. Then the cycle decomposition of $\pi$ is given by
$$\bigoplus_{d \mid N_2} \frac{\varphi(d) \cdot N_1}{\ord_{\eta(d)} P_{\pi}} \times \cyc\left( m \cdot \ord_{\eta(d)} P_{\pi}\right),$$
\end{theorem}

\begin{proof}
For each divisor $d$ of $N_2$, let $n_d$ be the number of cycles of $\pi$ containing an element $x = mx_0\in [1, n]$ such that 
the number $N_0=N_0(x_0)$ defined by Eq.~\eqref{N0} satisfies
$$N_0 = 
N_2/d.$$
Since every element of $[1,n]$ belongs to a unique cycle, Proposition~\ref{prop:length} and Corollary~\ref{cor:arithmetic} yield
$$\sum_{d \mid N_2} n_d \cdot m \cdot \ord_{\eta(d)}P_{\pi} = n.$$
We claim that $n_d \cdot \ord_{\eta(d)} P_{\pi} \ge \varphi(d) \cdot N_1$. In fact, by Corollary~\ref{cor:arithmetic}, for $x=mx_0$ the equality $N_0(x_0)= N_2/d$ implies that  $$\kappa(x)=N_1 \cdot \frac{P_{\pi}-1}{g_{\pi}} \cdot d = \eta(d),$$ 
where $\kappa(x)$ is given by Eq.~\eqref{kappa}. From the same corollary, the latter has exactly $\varphi(d)N_1$ solutions $x=mx_0\in [1, n]$. 
Since there exist at most $\ord_{\eta(d)} P_{\pi}$ of such $x=mx_0$ in a same cycle of length $m \cdot \ord_{\eta(d)} P_{\pi}$ of $\pi$, it follows that $n_d \ge \frac{\varphi(d) N_1}{\ord_{\eta(d)} P_{\pi}}$, proving the claim. 
Therefore, we obtain the following inequalities 
$$\dfrac{n}{m} = \sum_{d \mid N_2} n_d \cdot \ord_{\eta(d)} P_{\pi} \ge \sum_{d \mid N_2} \varphi(d)\cdot N_1 = N_1\sum_{d \mid N_2} \varphi(d)=N_1N_2 = \frac{n}{m},$$ 
forcing that $n_d = \dfrac{\varphi(d) N_1}{\ord_{\eta(d)} P_{\pi}}$.

\end{proof}

\begin{example}\label{ex:123again}
Let $\pi$ be the $(12, 3)$-p.a.p. given in Example~\ref{ex:1}. We have that $P_{\pi}=15$ and $S_{\pi}=9$, and in the notation of Theorem~\ref{thm:cycle}, $g_{\pi}=1$ and $N_1=4$, $N_2=1$. From Theorem~\ref{thm:cycle}, the cycle decomposition of $\pi$ is given by
$$\frac{4}{\ord_{56}15}\times \cyc(3\cdot \ord_{56}15)=2\times \cyc(6),$$
as confirmed by Example~\ref{ex:1}.

\end{example}

\section{Application: permutation polynomials over finite fields}\label{sec:pp}
Throughout this section, we fix $q$ a prime power and let $\F_q$ denote the finite field with $q$ elements. We observe that,  for a divisor $m>1$ of $q-1$, we may construct many $(q-1, m)$-p.a.p.'s. It turns out that such permutations extend to permutations of the finite field $\F_q$. Let $\theta_q\in \F_q$ be a {\em primitive element}, i.e. a generator of the multiplicative group $\F_q^*$. 
If $m>1$ divides $q-1$ and $\pi$ is any $(q-1, m)$-p.a.p., we define its {\em $\theta_q$-lift} as the permutation $F_{\pi,\theta_q}:\F_q\to \F_q$ given by 
$$\begin{cases}
F_{\pi,\theta_q}(0) \;\,\, = \; 0 \quad\quad\quad \,\, \text{ and } \\
F_{\pi,\theta_q}(\theta_q^i) \; = \, \theta_q^{\pi(i)} \quad\quad \text{ for any } 1\le i\le q-1.
\end{cases}$$
Of course, $F_{\pi, \theta_q}$ is a permutation of the finite field $\F_q$. We observe that, by construction, such permutation defines a piecewise monomial function on $m$-{\em cyclotomic cosets} of $\F_{q}^*$. In other words, if $\mathcal{D}_{m}\subset \F_q^*$ denotes the subgroup of perfect $m$-th powers, the restriction of $F_{\pi, \theta_q}(x)$ to each coset of $\F_{q}^*/\mathcal {D}_{m}$ is ruled by a monomial map $\alpha x^{\beta}$.

\begin{remark}
We emphasize that functions defined by different monomials on cyclotomic cosets of $\F_{q}^*$ were previously studied in full generality: see Theorem 2 of~\cite{Wa}. Our aim here is to apply our $(q-1, m)$-p.a.p.'s  in the construction of permutation polynomials where the cycle decomposition and the inverse can be obtained.
\end{remark}

We want to find a polynomial representation for $F_{\pi, \theta_q}$. Let $(\vec{a}, \vec{b}, \vec{c})$ be the parameters of $\pi$ with $\vec{a}=(a_1, \ldots, a_m)$, $\vec{b}=(b_1, \ldots, b_m)$ and $\vec{c}=(c_1, \ldots, c_m)$. In particular, if $x=\theta_q^{j}$ with $j\equiv c_i\pmod m$, then
$$F_{\pi, \theta_q}(x)=\theta_q^{b_i}\cdot x^{a_i}.$$
Therefore, we only need to find a characteristic function for the elements $x=\theta_q^{j}$ with $j\equiv c_i\pmod m$. We have the following definition.
\begin{define}
For each divisor $m$ of $q-1$, set $E_m(x) = \displaystyle\sum_{j=0}^{m-1}x^{\frac{(q-1)j}{m}} \in \F_q[x]$.
\end{define}
We observe that if $z \in \F_q^*$ then 
$$E_m(z) = 
\begin{cases} 
m \quad \text{ if } z^{\frac{q-1}m} = 1, \\
\, 0 \quad \, \text{ otherwise.}
\end{cases}$$
In particular, for each $j\in [1, m]$ we have that
\begin{equation}\label{Em}
E_m(z\cdot\theta_q^{-j}) = 
\begin{cases}
m \quad \text{ if $z = \theta_q^i$ with $i\equiv j\pmod m$,} \\
\, 0 \quad \, \text{ otherwise.}
\end{cases}
\end{equation}
The following theorem provides a polynomial representation for $F_{\pi, \theta_q}$ and its  inverse.

\begin{theorem}\label{thm:lift-F_q}
Let $q$ be a prime power, $\theta_q\in \F_q$ be a primitive element and $m>1$ be a divisor of $q-1$. If $\pi$ is a $(q-1, m)$-p.a.p. with parameters $(\vec{a}, \vec{b}, \vec{c})$, then the $\theta_q$-lift $F_{\pi, \theta_q}$ of $\pi$ admits the following polynomial representation
$$F_{\pi, \theta_q}(x)=\frac{1}{m}\sum_{i=1}^{m}\theta_q^{b_i}\cdot  x^{a_i}E_m(x\cdot \theta_q^{-c_i})\in \F_q[x].$$
In particular, this polynomial representation has at most $m^2$ nonzero coefficients. Moreover, if the $(q-1, m)$-p.a.p. $\pi^{-1}$ is the inverse of $\pi$ with parameters $(\vec{A}, \vec{B}, \vec{C})$ as in Theorem~\ref{thm:inverse}, then the  inverse of $F_{\pi, \theta_q}$ over $\F_q$ is the following permutation polynomial
$$F_{\pi^{-1}, \theta_q}(x)=\frac{1}{m}\sum_{i=1}^{m}\theta_q^{B_i}\cdot  x^{A_i}E_m(x\cdot \theta_q^{-C_i})\in \F_q[x].$$
\end{theorem}

\begin{proof}
Since $\pi$ is a $(q-1,m)$-p.a.p. with parameters $(\vec{a},\vec{b},\vec{c})$, we have that the function given by 
$F_{\pi, \theta_q}(0)=0$ and $F_{\pi, \theta_q}(y) = \theta_q^{a_i j + b_i} = \theta_q^{b_i} y^{a_i}$ if $y = \theta_q^j$ with $j\equiv c_i\pmod m$,
permutes $\F_q$. Since $E_m(x)/m$  acts as the characteristic function for the set of perfect $m$-th powers in $\F_q^*$, Eq.~\eqref{Em} entails that if $y = \theta_q^j$ with $j\equiv c_i\pmod m$, then
$$F_{\pi, \theta_q}(y)=\frac{1}{m}\sum_{i=1}^{m}\theta_q^{a_i j + b_i} E_m(y\cdot \theta_q^{-c_i}),$$
from where the polynomial expression for $F_{\pi, \theta_q}$ follows. The polynomial expression for the inverse of $F_{\pi, \theta_q}$ follows directly from the fact that the inverse $\pi^{-1}$ of $\pi$ is again a $(q-1, m)$-p.a.p. and that $F_{\pi, \theta_q}$ fixes $0\in \F_q$.

\end{proof}



\begin{remark}
We observe that if $\pi$ and $\pi_0$ are two $(q-1, m)$-p.a.p.'s coming from $(q-1, m)$-equivalent $(q-1, m)$-admissible triples, their $\theta_q$-lift coincide as permutations of $\mathbb F_q$. However, the polynomials $F_{\pi, \theta_q}$ and $F_{\pi_0, \theta_q}$ may not coincide and we can only guarantee that $$F_{\pi, \theta_q}(x)\equiv F_{\pi_0, \theta_q}(x)\pmod {x^q-x}.$$

\end{remark}

When $\pi$ is $2$-reduced, the following corollary entails that the polynomial representation of $F_{\pi, \theta_q}$ and its inverse are quite simple. Its proof is a direct application of the previous theorem so we omit details.

\begin{cor}\label{cor:lift-F_q}
Let $q$ be a prime power, $\theta_q\in \F_q$ be a primitive element and $m>1$ be a divisor of $q-1$. If $\pi$ is a $2$-reducible $(q-1, m)$-p.a.p. with reduced parameters $(a_0, a, b_0, b)$ and $\pi^{-1}$ is its  inverse, then the $\theta_q$-lift $F_{\pi, \theta_q}$ of $\pi$ admits the following polynomial representation
$$F_{\pi, \theta_q}(x) = x^a\theta_q^{b} + \left(\frac{x^{a_0}\theta_q^{b_0} - x^a\theta_q^{b}}{m}\right) E_m(x) \in \F_q[x],$$
whose inverse $F_{\pi^{-1}, \theta_q}$ is given by
$$F_{\pi^{-1}, \theta_q}(x) = x^{A} \theta_q^{B} + \left( \frac{x^{A_0} \theta_q^{B_0} - x^A \theta_q^B}{m} \right) E_m(x \cdot \theta_q^{-b}) \in \F_q[x],$$
where $A_0,A,B_0,B$ are defined in Proposition~\ref{prop:2-reducible}. In particular, these polynomial representations have at most $2m$ nonzero coefficients.
\end{cor}


\subsection{On the cycle decomposition}

We observe that the cycle decomposition of the permutation polynomials given in Theorem~\ref{thm:lift-F_q} can be explicitly computed. In fact, for a given $(q-1, m)$-p.a.p. $\pi$ with parameters $(\vec{a}, \vec{b}, \vec{c})$, we can compute its principal sum and product. In particular, the cycle decomposition of $\pi$ is explicitly obtained from Theorem~\ref{thm:cycle}. Moreover, for a fixed primitive element $\theta_q\in \F_q$, the cycle decomposition of the $\theta_q$-lift permutation $F_{\pi, \theta_q}$ is obtained by the one of $\pi$, adding a loop that corresponds to the fixed point $0\in \F_q$.  Furthermore, $F_{\pi,\theta_q}$ and its inverse $F_{\pi^{-1},\theta_q}$ have the same cycle decomposition. We provide two examples of these facts.


\begin{example}\label{ex:123againagain}
Let $q = 13$ and let $\pi$ be the $(12,3)$-p.a.p. given in Example~\ref{ex:1}. Applying Theorem~\ref{thm:lift-F_q} with $\theta_{13} = 2$, we obtain the permutation polynomial 
$$F_{\pi,2}(x) = 10x^{11} + 8x^9 + 12x^7 + x^5 + 4x^3 + 6x \in \F_{13}[x].$$
The  inverse $\pi^{-1}$ of $\pi$ is given in Example~\ref{ex:12,3}, whence we obtain the polynomial representation for the inverse of $F_{\pi, 2}$:
$$F_{\pi^{-1},2}(x) = 10x^{11}+8x^9+10x^7+4x^5+10x^3+x.$$
The cycle decomposition of $F_{\pi,2}$ (and of $F_{\pi^{-1},2}$) is given by $\cyc(1) \oplus (2 \times \cyc(6))$ (see also Example~\ref{ex:123again}).
\end{example}

\begin{example}\label{ex:q=25}
Let $q=25$ and let $\pi$ be the $2$-reducible $(24, 3)$-p.a.p. with reduced parameters $(5, 7, 2, 8)$ so that
$$\pi(x)=\begin{cases}\Psi_{24}(5x+2) & \text{if }\;\; x\equiv 0\pmod 3,\\
\Psi_{24}(7x+8)& \text{otherwise}.\end{cases}$$
Its  inverse $\pi^{-1}$ is given by
$$\pi^{-1}(x) = 
\begin{cases}
\Psi_{25}(5x+14) & \text{if }\;\; x \equiv 2 \pmod 3, \\
\Psi_{25}(7x+16) & \text{otherwise.}
\end{cases}$$
The principal product of $\pi$ equals $P_{\pi}=245$ and its principal sum equals $S_{\pi}=18$. In the notation of Theorem~\ref{thm:cycle}, we have that $g_{\pi}=2$ and $N_1=8$, $N_2=1$. From Theorem~\ref{thm:cycle}, the cycle decomposition of $\pi$ is given by
$$\frac{8}{\ord_{122\cdot 8}245}\times \cyc(3\cdot \ord_{122\cdot 8}245)=2\times \cyc(12).$$
Let $\F_{25}=\F_5(\alpha)$, where $\alpha^2-\alpha-3=0$, so that $\alpha$ is a primitive element. In particular, the $\alpha$-lift of $\pi$ is 
\begin{align*}
F_{\pi, \alpha}(x)&=x^7 \alpha^8 + \left( \frac{x^5 \alpha^2 - x^7 \alpha^8}{3} \right) E_3(x) \\ 
&= (\alpha+3)(x^{23}+x^{15}) + (2\alpha+1)(x^{21}+x^{13}-x^7+x^5),
\end{align*}
over $\F_{25}$, whose  inverse is 
\begin{align*}
F_{\pi^{-1},\alpha}(x) &= x^7 \alpha^{16} + \left( \frac{x^5 \alpha^{14} - x^7 \alpha^{16}}{3} \right) E_3(x \cdot \alpha^{-8}) \\
&=(\alpha + 3)x^{23} + 4x^{21} + 3x^{15} + (2\alpha + 2)(x^{13} + x^7) + (3\alpha + 4)x^5.
\end{align*}
The cycle decomposition of $F_{\pi,\alpha}$ (and of $F_{\pi^{-1},\alpha}$) is $\cyc(1)\oplus (2\times \cyc(12))$.
\end{example}

\subsubsection{Permutations yielding cycles of the same length}
Here we characterize the $(q-1, m)$-p.a.p.'s $\pi$ with the property that its cycles are of the same length $\ell$, a prime number. This is a nice application of Theorem~\ref{thm:cycle} and is stated as follows.

\begin{prop}\label{cic}
Let $\ell$ be a prime number, $q$ be a prime power, $\theta_q\in \F_q$ be a primitive element and $m$ be a divisor of $q-1$. If $\pi$ is a $(q-1, m)$-p.a.p. with principal product $P_{\pi}$ and principal sum $S_{\pi}$, then the $\theta_q$-lift $F_{\pi, \theta_q}$ of $\pi$ decomposes into cycles of length $\ell$ if and only if the following properties hold:
\begin{enumerate}[(i)]
\item $\ell=m$;
\item $P_{\pi}\equiv 1\pmod {\frac{q-1}{m}}$; 
\item $S_{\pi}\equiv 0\pmod{q-1}$.
\end{enumerate}
In this case, if $\pi^{-1}$ denotes the  inverse $\pi$, $F_{\pi,\theta_q}$ and $F_{\pi^{-1},\theta_q}$ have polynomial representations given by Theorem~\ref{thm:lift-F_q}, 
and the cycle decomposition of $F_{\pi, \theta_q}$ (and of $F_{\pi^{-1},\theta_q}$) over $\F_q$ equals $$\cyc(1)\oplus \left(\frac{q-1}{m}\times \cyc(m)\right).$$ 
\end{prop}

\begin{proof}
From construction, any $(q-1, m)$-p.a.p. decomposes into cycles of length divisible by $m$, forcing that $m=\ell$. Since $\ord_ba$ divides $\ord_ca$ whenever $\gcd(bc, a)=1$ and $b$ divides $c$, Theorem~\ref{thm:cycle} entails that $\pi$ has only cycles of length $m$ if and only if $P_{\pi}-1$ is divisible by
$$\frac{(q-1)(P_{\pi}-1)}{mg_{\pi}},$$
where $g_{\pi}=\gcd\left(\frac{S_{\pi}}{m}, P_{\pi}-1\right)$. The latter is equivalent to $g_{\pi}\equiv 0\pmod {\frac{q-1}{m}}$, i.e., $S_{\pi}\equiv 0\pmod{q-1}$ and $P_{\pi}\equiv 1\pmod {\frac{q-1}{m}}$. 

\end{proof}

In particular, it is possible to obtain involutions from $(q-1,2)$-p.a.p.'s. Involutions over finite fields are frequently used in cryptographic applications. More specifically, they are used as $S$-boxes, a basic component in key-algorithms used to cover the relation between the key and the encrypted message. We observe that if $P$ is an involution over $\F_q$, then any element $a\in \F_q$ either belongs to a cycle of length two or is a fixed point, i.e., $P(a)=a$. There are some cryptographic attacks that explore the number of fixed points of a permutation and according to \cite{B17},  for secure implementations, involutions should have few fixed points. In the particular case $m=2$ of Proposition~\ref{cic}, the $\theta_q$-lift $F_{\theta_q, \pi}$ recover a family of involutions that were previously obtained in~\cite{Wa1}. This is presented in the following corollary, which is just a straightforward application of the previous proposition. We omit details.

\begin{cor}
Let $q\equiv 3\pmod 4$ be a prime power and let $\theta_q\in \F_q$ be a primitive element. If $\pi$ is a $(q-1, 2)$-p.a.p. with parameters $\vec{a}=(a_0, a)$, $\vec{b}=(b_0, b)$ and $\vec{c}=(2, 1)$, then its $\theta_q$-lift $F_{\theta_q, \pi}$ is an involution if and only if $a_0a\equiv 1\pmod  {\frac{q-1}{2}}$ and $b_0a+b\equiv 1\pmod {q-1}$. In this case, the cycle decomposition of $F_{\pi, \theta_q}$ over $\F_q$ is given by $\cyc(1)\oplus \left(\frac{q-1}{2}\times \cyc(2)\right),$ 
and so it has only one fixed point. Moreover, in this case, $F_{\pi, \theta_q}$ has the following polynomial representation
$$F_{\pi,\theta_q}(x) = \theta_q^{b_0}\cdot \frac{x^{\frac{q-1}{2}+a_0}+x^{a_0}}{2}  + \theta_q^{b}\cdot\frac{x^a-x^{\frac{q-1}{2}+a}}{2}.$$
\end{cor}

\begin{example}
Let $q=27$ and let $\pi$ be the $2$-reducible $(26, 2)$-p.a.p. with reduced parameters $(5, 8, 3, 2)$ so that
$$\pi(x)=\begin{cases}\Psi_{26}(5x+3) & \text{if $x$ is even},\\
\Psi_{26}(8x+2) & \! \text{ if $x$ is odd}.\end{cases}$$
Let $\F_{27}=\F_3(\alpha)$ where $\alpha^3-\alpha-2=0$, so $\alpha$ is a primitive element. In particular, the $\alpha$-lift of $\pi$ yields the involution
\begin{align*}
F_{\pi, \alpha}(x)&=\alpha^3 \left( \dfrac{x^{18} + x^5}{2} \right) + \alpha^2 \left( \dfrac{x^8 - x^{21}}{2} \right) \\ 
&= \alpha^2 (x^{21} - x^8) - (\alpha+2) (x^{18} + x^5),
\end{align*}
over $\F_{27}$, whose cycle decomposition is $\cyc(1)\oplus \left(13\times \cyc(2)\right)$.
\end{example}

\subsection{More explicit results}

We observe that Theorem~\ref{thm:lift-F_q} provides classes of permutation polynomials over $\F_q$, that depend on a primitive element $\theta_q\in \F_q$. If $q$ is large, it can be hard to find such a $\theta_q$. Here we consider special cases where such permutation polynomials can be obtained without going through a primitive element of $\F_q$. Instead, we only need certain primitive roots of unity in the base field $\F_p$ of $\F_q$. This is done in the following proposition.

\begin{prop}
Let $p$ be a prime and $m, k$ be positive integers such that $m$ divides $p-1$. 
Write $p^k-1=n_1n_2$,  where $\rad(n_1)$ divides $\frac{p^k-1}{m}$ and $\gcd\left(n_2, \frac{p^k-1}{m}\right)=1$. Let $\theta\in \F_p$ be any primitive $m$-th root of unity, write $q=p^k$ and let $a, a_0$ be positive integers such that $\gcd(aa_0, n_1)=1$ and $a\equiv 1\pmod {\rad_2(m)}$. Then, for any positive integer $b<m$ such that $\gcd(b, m)=1$,
$$F_{a_0, a, b}(x) = \theta^{b}\left(\frac{1}{m}E_m(x) \cdot x^{a_0} + \left(1 - \frac{1}{m}E_m(x) \right) \cdot x^a\right)$$
is a permutation polynomial over $\F_q$. Set $g=\gcd\left(\frac{q-1}{m} , a_0a^{m-1} - 1 \right)$ and write $\frac{q-1}{m}=N_1N_2$, where $\rad(N_1)$ divides $\frac{a_0a^{m-1}-1}{g}$ and $\gcd\left( \frac{a_0a^{m-1}-1}{g},N_2 \right) = 1$. Then the cycle decomposition of the permutation polynomial $F_{a_0, a, b}$ over $\F_q$ is given by
\begin{equation}\label{eq:cycdecom}\cyc(1)\oplus \left(\bigoplus_{d \mid N_2} \frac{\varphi(d) \cdot N_1}{\ord_{\eta(d)} (a_0a^{m-1})} \times \cyc\left( m \cdot \ord_{\eta(d)} (a_0a^{m-1})\right)\right),\end{equation}
where $\eta(d) = N_1 \cdot \frac{a_0a^{m-1} - 1}{g} \cdot d$.
\end{prop}

\begin{proof}
 From construction and Proposition~\ref{prop:2-reducible}, $(a_0, a, B, B)$ are the reduced parameters of a $(q-1, m)$-p.a.p., where $B=\frac{b(q-1)}{m}$. Therefore, as $\theta_q$ runs over the primitive elements of $\F_q$, $\theta_q^{\frac{q-1}{m}}$ runs over the primitive $m$-th roots of unity in $\F_p$. In particular, the fact that $F_{a_0, a, b}(x)$ permutes $\F_q$ follows from Theorem~\ref{thm:lift-F_q}. Let $\pi$ be the $(q-1, m)$-p.a.p. with reduced parameters  $(a_0, a, B, B)$ where $B$ is as before. Therefore, $\pi$ has principal product $P_{\pi}=a_0a^{m-1}$ and principal sum $S_{\pi}=q-1$. In particular, Eq.~\eqref{eq:cycdecom} follows from Theorem~\ref{thm:cycle}.

\end{proof}

Some cases of the previous proposition readily yield explicit results.

\begin{cor}
Let $q$ be a prime power such that $q\equiv 3\pmod 4$ and  let $a, a_0$ be positive integers such that $\gcd\left(aa_0, \frac{q-1}{2}\right)=1$ and $a\equiv 1\pmod {4}$. Then 
$$P_{a_0, a}(x)=-\frac{x^{\frac{q-1}{2}}+1}{2}\cdot x^{a_0} +\frac{x^{\frac{q-1}{2}}-1}{2} \cdot x^a,$$
is a permutation polynomial over $\F_q$ with cycle decomposition given by Eq.~\eqref{eq:cycdecom}.
\end{cor}

\begin{cor}
Let $q=7^k$ with $\gcd(k, 3)=1$ and  let $a, a_0$ be positive integers such that $\gcd\left(aa_0, \frac{q-1}{3}\right)=1$ and $a\equiv 1\pmod {3}$. Then for $j=1, 2$, 
$$P_{a_0, a, j}(x)= 2^j\left(\frac{1+x^{\frac{q-1}{3}}+x^{\frac{2(q-1)}{3}}}{3} \cdot x^{a_0} + \left(1 - \frac{1+x^{\frac{q-1}{3}}+x^{\frac{2(q-1)}{3}}}{3} \right) \cdot x^a\right),$$
is a permutation polynomial over $\F_q$ with cycle decomposition given by Eq.~\eqref{eq:cycdecom}.
\end{cor}

\section*{Acknowledgments}
The first author was supported by FAPESP under grant 2018/03038-2.


\end{document}